\documentclass[10pt,a4paper,notitlepage]{amsart}

\usepackage[spanish,USenglish]{babel} 

\usepackage{amsthm}

\usepackage{mathrsfs}

\usepackage{graphicx}

\usepackage{amsmath}

\usepackage{color}

\usepackage{amssymb}

\newcommand{\R}{\mathbb{{R}}}

\newcommand{\N}{\mathbb{{N}}}

\newcommand{\C}{\mathbb{{C}}}
\newcommand{\M}{\mathbb{{M}}}

\newtheorem{theorem}{Theorem}
\newtheorem{corollary}[theorem]{Corollary}
\newtheorem{lemma}[theorem]{Lemma}
\newtheorem{example}[theorem]{\it Example}
\newtheorem{proposition}[theorem]{Proposition}
\newtheorem{definition}[theorem]{Definition}

\newtheorem{remark}[theorem]{\it Remark}

\begin{document}
\title[A Subordination Principle on Regularized Resolvent Families]
{A Subordination Principle on  Wright Functions  and Regularized Resolvent Families}

\author{Luciano Abadias}
\address{Departamento de Matem\'aticas, Instituto Universitario de Matem\'aticas y Aplicaciones, Universidad de Zaragoza, 50009 Zaragoza, Spain.}

\email{labadias@unizar.es}

\author{Pedro J. Miana}
\address{Departamento de Matem\'aticas, Instituto Universitario de Matem\'aticas y Aplicaciones, Universidad de Zaragoza, 50009 Zaragoza, Spain.}
\email{pjmiana@unizar.es}

\thanks{Authors have been partially supported  by Project MTM2013-42105-P, DGI-FEDER, of the MCYTS; Project E-64, D.G. Arag\'on, and  Project UZCUD2014-CIE-09, Universidad de Zaragoza.}



\begin{abstract}We obtain a vector-valued subordination principle for $(g_{\alpha},g_{\beta})$-regularized resolvent families which unified and  improves various previous results in the literature.  As a consequence we establish new relations between  solutions of different fractional Cauchy problems.
To do that, we consider   scaled Wright functions which are related to Mittag-Leffler functions, the fractional calculus and stable L\'{e}vy processes. We study some interesting properties of these functions such as subordination (in the sense of Bochner), convolution properties, and their Laplace transforms. Finally we present some examples where we apply these results.
\end{abstract}
\date{}

\maketitle

\section{Introduction}

\setcounter{theorem}{0}
\setcounter{equation}{0}
A function $f:(0,\infty)\to \R$ is  a Bernstein function if $f$ is of class $C^\infty$, $f(\lambda)\ge 0 $ for all $\lambda>0$ and $$(-1)^{n-1}f^{(n)}(\lambda)\ge 0,  \qquad  \lambda>0, \,\, n\in \N.
$$
The celebrated Bochner subordination theorem characterizes Bernstein functions: given $f$ a Bernstein function there exists a unique convolution semigroup of sub-probability measures $(\mu_t)_{t>0}$ on $[0, \infty)$ such that
\begin{equation}\label{Berstein}
e^{-tf(\lambda)}=\int_0^\infty e^{-\lambda s}d\mu_t(s), \qquad \Re\lambda>0.
\end{equation}
Conversely, given a convolution semigroup of sub-probability measures  $(\mu_t)_{t>0}$ on $[0, \infty)$, then there exists a unique Bernstein function $f$ such that (\ref{Berstein}) holds true, see for example \cite[Theorem 5.2]{Berstein}. The original subordination principle for stochastic processes  in connection with diffusion equations and semigroups  was introduced  in \cite{Bochner}. In \cite[Chapters 4.3, 4.4]{Bochnerbook}  a detailed study of stochastic processes, their transition semigroups, generators and subordination results are developed.

Now let $A$ be a densely defined closed linear operator on a Banach space $X$ which generates a  $C_0$-contraction semigroup $(T(t))_{t>0}\subset {\mathcal B}(X)$. Then the solution of the first order abstract Cauchy problem
$$
\left\{\begin{array}{ll}
u'(t)=Au(t),&t>0, \\
u(0)=x,&
\end{array} \right.
$$
is given by $u(t)=T(t)x$ for $t>0$. Now, suppose that $(\mu_t)_{t>0}$ is a vaguely continuous convolution semigroup of sub-probability measures on $[0, \infty)$ with the corresponding Bernstein function $f$. Then the Bochner integral
\begin{equation*}
T^f(t)x:=\int_0^\infty T(s)x\,d\mu_t(s), \qquad x\in X, \quad t>0,
\end{equation*}
defines again a  $C_0$-contraction semigroup on $X$ (\cite[Proposition 12.1]{Berstein}). Then the semigroup $(T^f(t))_{t>0}$
is called subordinate (in line with Bochner) to the semigroup
$(T(t))_{t>0}$ with respect to the Bernstein function $f$. In particular, given $0<\alpha<1$ and $d\mu_t(s)= f_{t,\alpha}(s)ds$ (where $f_{t,\alpha}$ are the stable L\'{e}vy processes, see (\ref{Levi})) then
\begin{equation}\label{yosida}
{T}^{(\alpha)}(t)x:=
\displaystyle\int_0^{\infty}f_{t,\alpha}(s)T(s)x\,ds, \qquad x\in X,\quad t>0,
\end{equation}
is an analytic semigroup generated by $-(-A)^{\alpha},$ the fractional powers of the generator $A$
according to Balakrishnan. For more details see \cite[Chapter IX]{Yosida}.

Other subordination formulae allow to define new families of operators from some previous ones by integration. Let $A$ be the generator of a cosine function $(C(t))_{t>0}$ on a Banach space $X$ (see definition in \cite[Section 3.14]{ABHN}). Then $A$ generates a holomorphic $C_0$-semigroup $(T(z))_{z\in \C_+}$ of angle ${\pi \over 2}$, given by
\begin{equation}\label{cosine}
T(z)x={1\over \sqrt{\pi z}}\int_0^\infty e^{-s^2\over 4z}C(s)x\,ds, \qquad x\in X, \quad z\in \C_+,
\end{equation}
(\cite[Theorem 3.14.17]{ABHN}). Remember that the solution of the second order Cauchy problem
$$
\left\{\begin{array}{ll}
u''(t)=Au(t),&t>0, \\
u(0)=x,&\\
u'(0)=0,&
\end{array} \right.
$$
is $u(t)=C(t)x$ for $t>0$ (\cite[Section 3.14]{ABHN}).

In \cite{Lizama}, a two-kernel dependent family of strong continuous operators
defined in a Banach space is introduced. This family allows us to consider in a unified treatment the
notions of, among others, $C_0$-semigroups of operators, cosine families, $n$-times
integrated semigroups, resolvent families and $k$-generalized solutions. Let $a\in L^1_{loc}(\R_+)$ and $k\in C(\R_+)$. The family $\{S_{a,k}(t)\}_{t> 0}\subset\mathcal{B}(X)$ is a $(a,k)$-regularized resolvent family generated by $A$ if the following conditions are fulfilled:  $S_{a,k}(t)$ is strongly continuous for $t> 0$ and
$\displaystyle{S_{a,k}(0)x}={k(0)x}$ for all $x\in X;$  $S_{a,k}(t)A\subset AS_{a,k}(t)$, i.e.,  $S_{a,k}(t)A(x)\subset AS_{a,k}(t)x$ for $x\in D(A)$ and  $t>0;$ and
 $$S_{a,k}(t)x=k(t)x+A\left(\int_0^t{a(t-s)}S_{a,k}(s)x\,ds\right),\qquad x\in X, \quad t>0 ,$$
see \cite[Definition 2.1]{Lizama}. In the case $k(t)=1$, we obtain the resolvent families which are treated in detail in \cite{pruss}; for $k(t)=a(t)=1$, this family of operators is a  $C_0$-semigroup; and we also retrieve cosine functions  for $k(t)=1$ and $a(t)=t$, $(t>0)$.

Subordination theorems for $(a,k)$-regularized resolvents have been considered in some different works. In \cite[Section I.4]{pruss},  the use of the theory of Bernstein functions, completely positive functions and the Laplace transform allow to show some subordination results for $(a, 1)$-regularized resolvents (\cite[Theorem 4.1, Corollary 4.4 and Corollary 4.5]{pruss}). A nice subordination theorem for $m$-times integrated semigroup is proved in \cite[Theorem 3.7]{Lizama}. In the case of $({t^{\alpha-1}\over \Gamma(\alpha)},1)$-regularized functions, this subordination theorem is improved in  \cite[Theorem 3.1]{Bazhlepaper} and \cite[Theorem 3.1]{Bajlekova} and an integral representation, similar to formula (\ref{cosine}), is also proved involving Wright functions. In \cite[Theorem 3.1]{Li}, using holomorphic functional calculus, the authors prove a subordination result for $(\frac{t^{\alpha-1}}{\Gamma(\alpha)},1)$-regularized resolvent families generated by fractional powers of closed operators, which extends both \cite[Chapter IX, section 11, Theorem 2]{Yosida} and \cite[Theorem 3.1]{Bazhlepaper}. Finally, in \cite[Theorem 2.8]{Kostic} a subordination principle for $(a,k)$-regularized resolvents, inspired in the original proof of \cite[Theorem 4.1]{pruss}, is shown. In all this results, note that the subordination integral formula is only present in \cite[Theorem 3.1]{Bazhlepaper} and \cite[Theorem 3.1]{Li}.

The main aim of this paper is to obtain subordination integral formulae to  $({t^{\alpha-1}\over \Gamma(\alpha)},{t^{\beta-1}\over \Gamma(\beta)})$-regularized resolvents (Theorem \ref{main}). To achieve this, we present a detailed presentation of Wright and Mittag-Leffler functions in Section 2, which includes some  basic results and known connections of these functions and fractional differential equations.

In Section 3, we introduce a new family of bi-parameter special functions $\psi_{\alpha, \beta}$ in two variables defined by scaling Wright functions (Definition \ref{psi}). This family of functions $\psi_{\alpha, \beta}$ plays a fundamental role in the subordination principle for $({t^{\alpha-1}\over \Gamma(\alpha)},{t^{\beta-1}\over \Gamma(\beta)})$-regularized resolvent families, see formula (\ref{formulasub}). Moreover, these functions satisfy a nice subordination formula,  Theorem \ref{bonito}, which extends some known results for  Wright $\M$-function and stable L\'{e}vy processes, see Remark  \ref{cases}. In fact the algebraic nature (for convolution products) of these functions   $\psi_{\alpha, \beta}$ is shown in Proposition \ref{convo} and \ref{integrales}.

In Section 4, we consider $({t^{\alpha-1}\over \Gamma(\alpha)},{t^{\beta-1}\over \Gamma(\beta)})$-regularized resolvents in abstract Banach spaces. We prove spectral inclusions in Theorem \ref{spectral}, which extends the spectral inclusion result for $({t^{\alpha-1}\over \Gamma(\alpha)},1)$-regularized resolvents, see \cite[Theorem 3.2]{Li-Zheng};  the subordination principle, Theorem \ref{main}, and some consequences in Remark \ref{others} and Corollary \ref{corollary}.

Finally, in Section 5 we present some  comments, concrete examples and applications to fractional Cauchy problems which illustrate the main results of this paper .

\noindent {\bf Notation.} Let $\R_+:=[0,\infty),$ $\C_+:=\{z\in\C\,:\, \mathfrak{R}z>0\},$ and $L^1(\R_+)$ be the Lebesgue Banach algebra of integrable function on $\R_+$ with the usual convolution product $$f*g(t) =\displaystyle\int_0^t f(t-s)g(s)\,ds,\qquad f,g\in L^1(\R_+),\ t\geq 0.$$
The usual Laplace transform  of a function $f$, $\hat f$, is defined by
$$
\hat f(\lambda)=\int_0^\infty f(t)e^{-\lambda t}dt,\qquad \lambda \in \C_+,
$$
for $f\in L^1(\R_+).$ Let $\gamma>0,$ we denote by $g_{\gamma}(t):=\frac{t^{\gamma-1}}{\Gamma(\gamma)},$ $t>0,$ and $\widehat{g_{\gamma}}=\displaystyle{\frac{1}{\lambda^{\gamma}}}$ for $\lambda\in \C_+.$

The set of continuous functions on $\R_+$ such that $\displaystyle\lim_{t\to\infty}|f(t)|=0$ is denoted by $C_0(\R_+),$ and the set of holomorphic functions on $\C_+$ such that $\displaystyle\lim_{|z|\to\infty}|f(z)|=0$ by ${\mathcal H}_0(\C^+).$ We denote by $X$ an abstract Banach space, ${\mathcal B}(X)$ the set of linear and bounded operators on the Banach space $X$, and $C_{c}^{(\infty)}(\R_+; X)$ the set of functions of compact support and infinitely differentiable on $\R_+$ into $X$.

\section{Mittag-Leffler and Wright functions}

\setcounter{theorem}{0}
\setcounter{equation}{0}

In this section we present definitions and basic properties of Mittag-Leffler and Wright functions. The algebraic structure of those functions have  been partially considered in \cite{Peng-Li-10} and formulae (\ref{eq2.3}) and (\ref{eq2.4}) seems to be new.

The Mittag-Leffler functions are defined by $$E_{\alpha,\beta}(z):=\displaystyle\sum_{n=0}^{\infty}\frac{z^n}{\Gamma(\alpha n+\beta)}, \qquad \alpha,\,\beta>0,\,z\in\C.$$ We write $E_{\alpha}(z):=E_{\alpha,1}(z).$ The Mittag-Leffler functions satisfy the following fractional differential problems $$_C D_t^{\alpha}E_{\alpha}(\omega t^{\alpha})=\omega E_{\alpha}(\omega t^{\alpha}),$$ and $$_R D_t^{\alpha}\biggl(t^{\alpha-1}E_{\alpha,\alpha}(\omega t^{\alpha})\biggr)=\omega t^{\alpha-1}E_{\alpha,\alpha}(\omega t^{\alpha}),$$ for $0<\alpha<1,$ under certain initial conditions, where $_C D_t^{\alpha}$ and $_R D_t^{\alpha}$ denote the Caputo and Riemann-Liouville fractional derivatives of order $\alpha$ respectively, see section 5 and \cite{Mainardi, Miller}. Their Laplace transform is
\begin{equation}\label{mittag}\int_0^{\infty}e^{-\lambda t}t^{\beta-1}E_{\alpha,\beta}(\omega t^{\alpha})\,dt=\frac{\lambda^{\alpha-\beta}}{\lambda^{\alpha}-\omega}, \qquad \mathfrak{R}\lambda>\omega^{\frac{1}{\alpha}},\,\omega>0.\end{equation} For more details see \cite[Section 1.3]{Bajlekova}.

Recently the next algebraic property has been proved
  \begin{equation*}
 \left( \int_t^{t+s}  - \int_0^s\right) \frac{E_\alpha(\omega r^\alpha)}{(t+s-r)^{\alpha}}\, dr = \alpha \int_0^t \int_0^s \frac{E_\alpha(\omega r^\alpha_1)E_\alpha(\omega r^\alpha_2)}{(t+s-r_1-r_2)^{1+\alpha}}\,dr_1\, dr_2, \qquad t,s\ge 0,
 \end{equation*}
 for $0<\alpha <1$ and $\omega \in \C$, see  \cite[Theorem 1]{Peng-Li-10}. In fact, a similar identity holds for generalized Mittag-Leffer function $E_{\alpha, \beta}$ with $0<\alpha<1$, $\beta>\alpha$ and
 \begin{equation}\begin{array}{l}\label{eq2.3}
 \displaystyle\left( \int_t^{t+s}  - \int_0^s\right) {(t+s-r)^{\beta-\alpha-1}\over \Gamma(\beta- \alpha)}r^{\beta-1}E_{\alpha, \beta}(\omega r^\alpha)\, dr \\ \\
 \displaystyle = {\alpha\over \Gamma(1-\alpha)} \int_0^t \int_0^s \frac{r_1^{\beta-1}E_{\alpha,\beta}(\omega r^\alpha_1)r_2^{\beta-1}E_{\alpha,\beta}(\omega r^\alpha_2)}{(t+s-r_1-r_2)^{1+\alpha}}\,dr_1\, dr_2,
 \end{array}\end{equation}
for $t,s\ge 0$ and $\omega>0$. The proof of this result is a straightforward consequence of the \cite[Theorem 5]{Li-Sun1}. In the case $\beta=\alpha$ for $0<\alpha<1,$ the algebraic property is \begin{equation}\label{eq2.4}
 (t+s)^{\alpha-1}E_{\alpha, \alpha}(\omega (t+s)^\alpha)
 = {\alpha\over \Gamma(1-\alpha)} \int_0^t \int_0^s \frac{r_1^{\alpha-1}E_{\alpha,\alpha}(\omega r^\alpha_1)r_2^{\alpha-1}E_{\alpha,\alpha}(\omega r^\alpha_2)}{(t+s-r_1-r_2)^{1+\alpha}}\,dr_1\, dr_2,
 \end{equation}
which is a direct consequence of Theorem 2.1 and Theorem 2.2 of \cite{Mei-Pe-Zh13}.

The Wright function, that we denote by $W_{\lambda,\mu},$ was introduced and investigated by E. Maitland Wright in a series of notes starting from 1933 in the framework of the theory of partitions, see \cite{Wright}. This entire function is defined by the series representation, convergent in the whole complex plane, $$W_{\lambda,\mu}(z):=\displaystyle\sum_{n=0}^{\infty}\frac{z^n}{n!\Gamma(\lambda n+\mu)},\qquad \lambda>-1,\ \mu\in\C.$$

The equivalence between the above series and the following integral representations of $W_{\lambda,\mu}$ is easily proven by using the Hankel
formula for the Gamma function, $$W_{\lambda,\mu}(z)=\frac{1}{2\pi i}\int_{Ha}\sigma^{-\mu}e^{\sigma+z\sigma^{-\lambda}}\,d\sigma,\qquad \lambda>-1,\ \mu\ge 0,\, z\in \C,$$ where $Ha$ denotes the Hankel path defined as a contour that begins at $t = -\infty- ia$ $(a > 0),$ encircles the branch cut that lies along the negative real axis, and ends up at $t = -\infty+ ib$ $(b > 0),$ for more details see \cite[Appendix F]{Mainardi}. It is clear that $$\frac{d}{dz}W_{\lambda,\mu}(z)=W_{\lambda,\lambda+\mu}(z), \qquad z\in \C.$$

In addition, as discussed below, the following special cases are of considerable interest: $$M_{\alpha}(z):=W_{-\alpha,1-\alpha}(-z), \qquad F_{\alpha}(z):=W_{-\alpha,0}(-z),\qquad 0<\alpha<1,\, z\in\C,$$ interrelated through $$F_{\alpha}(z)=\alpha z M_{\alpha}(z), \qquad z\in \C.$$


The Wright $\M$-function in two variables $\M$ is defined by
$$ \M_\alpha(s,t): =t^{-\alpha}M_{\alpha}(st^{-\alpha}), \qquad t>0, s\in \R.$$ This function has been studied, for example, in \cite[p. 257]{Mainardi} and  \cite[Section 6]{MaiPaGo}; a subordination formula for time fractional diffusion process is given in \cite[Formula (6.3)]{MaiPaGo} and \cite[(F.55)]{Mainardi}: for $ \eta, \beta\in (0,1)$, the following subordination formula holds true for $0<s, t$,
\begin{equation}\label{Mainardi}
\M_{\eta\beta}(s,t)= \int_0^\infty \M_\eta(s,\tau)\M_\beta(\tau,t)d\tau, \qquad t,s>0.
\end{equation}
This subordination formula had previously appeared  in \cite[Formula (3.28)]{Bajlekova}.

The deep connection between fractional differential equations (in space and in time) and Wright-type functions $(W_{\lambda,\mu},\, M_{\alpha}, \,F_{\alpha},\, \M_\alpha...)$ has been studied in detail in  \cite{MaiLuPa, MaiPaGo, Mainardi}.

It is known that \begin{equation}\label{LTWF}E_{\alpha,\alpha+\beta}(z)=\int_0^{\infty}e^{zt}\,W_{-\alpha,\beta}(-t)\,dt,\qquad z\in\C,\,0<\alpha<1,\,\beta\geq 0,\end{equation} that is, $E_{\alpha,\alpha+\beta}(-(\cdot))$ is the Laplace transform of $W_{-\alpha,\beta}(-(\cdot))$ in the whole complex plane, see \cite[Formula (F.25)]{Mainardi}. Then, observe that for $0<\alpha<1$ $$E_{\alpha}(z)=\displaystyle\int_0^{\infty}e^{zt}M_{\alpha}(t)\,dt,\qquad E_{\alpha,\alpha}(z)=\displaystyle\int_0^{\infty}e^{zt}F_{\alpha}(t)\,dt,\qquad z\in \C,$$ where both functions are related to the solutions of the fractional differential problems mentioned above.

Nice connections between Mittag-Leffler functions and Wright functions are obtained by the Laplace  transform, see formula (\ref{LTWF}) and
$$
\int_0^\infty e^{-zt}W_{\lambda, \mu}(\pm r)dr={1\over z}E_{\lambda,\mu}\left(\pm{1\over z}\right), \qquad \vert z\vert >0,\quad \lambda >0,
$$
(\cite[Formula (F.22)]{Mainardi}). In the next proposition, we present some interesting properties of Wright functions.  The next result extends the study which was done in \cite[Chapter 1, p.14]{Bajlekova} for the case  $W_{-\alpha, 1-\alpha}$ with $0<\alpha<1$.

\begin{proposition}\label{W} Let $0<\alpha<1$ and $\beta\geq 0.$ Then the following properties hold.
\begin{itemize}
\item[(i)] $\displaystyle\int_0^{\infty}{t^{\eta-1}\over\Gamma(\eta)}W_{-\alpha,\beta}(-t)\,dt=\frac{1}{\Gamma(\alpha\eta+\beta)},\qquad \eta>0.$
\item[(ii)] $W_{-\alpha,\beta}(-t)\geq 0,$ for $t>0.$
\end{itemize}
\end{proposition}
\begin{proof} (i) Using the definition of $W_{-\alpha,\beta}$ we have $$\int_0^{\infty}{t^{\eta-1}\over\Gamma(\eta)}W_{-\alpha,\beta}(-t)\,dt=\frac{1}{2\pi i}\int_{H_a}e^{\sigma}\sigma^{-\beta-\alpha\eta}\,d\sigma=g_{\beta+\alpha\eta}(1),$$ where we have applied the Fubini theorem and the Laplace transform of $g_{\eta}.$

(ii) The positivity of $W_{-\alpha,\beta}(-t)$ follows from \eqref{LTWF}, the complete monotonicity of $E_{\alpha,\gamma}(-t)$ for $t>0,$ $0<\alpha<1$ and $\gamma\geq \alpha,$ see \cite[Appendix E, formula (E.32)]{Mainardi}, and the Post-Widder inversion formula, see \cite[Lemma 1.6]{Bajlekova}.
\end{proof}

\section{Scaled Wright functions}

\setcounter{theorem}{0}
\setcounter{equation}{0}

In this section, we introduce  two-parameter Wright functions in Definition \ref{psi}, which we call scaled Wright functions. This  class of functions includes the Wright {$\M$}-function introduced in \cite[Formula (6.2)]{MaiPaGo} and  also considered in \cite[(F.51)]{Mainardi} and stable L\'{e}vy processes. They satisfy important properties (Theorem \ref{propiedades} and Proposition \ref{convo}), a subordination principle (Theorem \ref{bonito}) and  play a crucial role in this paper.

\begin{definition}\label{psi} {\rm For $0<\alpha<1$ and $\beta\geq 0,$ we define the function  $\psi_{\alpha,\beta}$ in two variables by
 $$\psi_{\alpha,\beta}(t,s):=t^{\beta-1}W_{-\alpha,\beta}(-st^{-\alpha}),\qquad t>0,\ s\in\C.$$}
\end{definition}

\noindent Note that using the change of variable $z=\frac{\sigma}{t},$ we get the integral representation $$\psi_{\alpha,\beta}(t,s)=\frac{1}{2\pi i}\int_{Ha}z^{-\beta}e^{tz-sz^{\alpha}}\,dz, \qquad t,s>0.$$

The function $\psi_{\alpha,\beta}$ is considered in the literature in some particular cases:
 \begin{itemize}
 \item[(i)]
 for $\beta=1-\alpha,$ $$\psi_{\alpha,1-\alpha}(t,s)=t^{-\alpha}M_{\alpha}(st^{-\alpha})=\M_\alpha(s,t)=\varphi_{t,\alpha}(s), \qquad t,s>0,$$ where $\M_\alpha(s,t)$ is
the Wright $\M$-function in two variables studied in \cite[p. 257]{Mainardi} and $\varphi_{t,\alpha}(s)$ is considered in \cite[Theorem 3.1]{Bajlekova}; for $\alpha=\frac{1}{2},$ \begin{equation}\label{gassuian}\psi_{\frac{1}{2},\frac{1}{2}}(t,s)=\frac{1}{\sqrt{\pi t}}e^{-\frac{s^2}{4t}}, \qquad t,s>0, \end{equation} see \cite[Appendix F, formula (F.16)]{Mainardi}.

\item[(ii)] for $\beta=0,$
\begin{equation}\label{Levi}\psi_{\alpha,0}(t,s)=\frac{1}{2\pi i}\int_{Ha}e^{tz-sz^{\alpha}}\,dz=:f_{s,\alpha}(t), \qquad t,s>0\end{equation} is the stable L\'{e}vy process of order $\alpha$, see introduction, \cite{Bochner} and  \cite[Chapter IX]{Yosida}, in particular
\begin{equation*}
\psi_{{1\over 2},0}(t,s)={1\over 2\sqrt{\pi}}t^{-3\over 2}se^{-s^2\over 4t}, \qquad t,s>0.
\end{equation*}
\end{itemize}

In the next proposition, we join some properties which are verified by functions $\psi_{\alpha, \beta}$.

\begin{theorem}\label{propiedades} Let $0<\alpha<1$ and $\beta\geq 0,$ we have \begin{itemize}
\item[(i)] $\psi_{\alpha,\beta}(t,s)\geq 0,$ for $t,s>0.$

\item[(ii)]$\displaystyle\int_0^{\infty}e^{-\lambda t}\psi_{\alpha,\beta}(t,s)\,dt=\lambda^{-\beta}e^{-\lambda^{\alpha}s},$ for $ s, \lambda>0.$

\item[(iii)]$\displaystyle\int_0^{\infty}e^{\lambda s}\psi_{\alpha,\beta}(t,s)\,ds=t^{\alpha+\beta-1}E_{\alpha,\alpha+\beta}(\lambda t^{\alpha}),$ for $t>0,\ \lambda\in\C.$

\item[(iv)]$\displaystyle\int_0^{\infty}\int_0^{\infty}e^{-\lambda s}e^{-\mu t}\psi_{\alpha,\beta}(t,s)\,ds\,dt={1\over \mu^{\beta}(\mu^\alpha+\lambda)}$ for $t>0,\ \lambda,\mu>0.$

\item[(v)]$\psi_{\alpha,\beta+\gamma}(t,s)=(g_{\gamma}*\psi_{\alpha,\beta}(\cdot,s))(t),$ for $t,s, \gamma>0.$

\item[(vi)] $\displaystyle\int_0^{\infty}g_{\eta}(s)\psi_{\alpha,\beta}(t,s)\,ds=g_{\alpha\eta+\beta}(t),$ for $t, \eta>0.$

\end{itemize}

\end{theorem}
\begin{proof}

\noindent (i) It is clear by Definition \ref{psi} and Proposition \ref{W}. (ii) It is easy to see that $\psi_{\alpha,\beta}(t,s)$ is of exponential growth in $t$. Let $\lambda>\max(a,b)>0,$ (where $a$ and $b$ are involved in the definition of Hankel path $Ha$). Then by Cauchy' theorem of residue, \begin{eqnarray*}
\displaystyle\int_0^{\infty}e^{-\lambda t}\psi_{\alpha,\beta}(t,s)\,dt=\frac{1}{2\pi i}\int_{Ha}\frac{z^{-\beta}e^{-sz^{\alpha}}}{z-\lambda}\,dz=\lambda^{-\beta}e^{-s\lambda^{\alpha}}.
\end{eqnarray*} (iii) Using \eqref{LTWF}, the result is direct by a change of variable. (iv) We combine part (ii) and (iii) to obtain the equality. (v) It is clear using Laplace transform and (i).
(vi) It is clear by a change of variable and applying Proposition \ref{W} (i).
\end{proof}

We combine  Theorem \ref{propiedades} (ii) and (iii) and formula (\ref{mittag}) to get the following corollary.

\begin{corollary}\label{PsiCapital} For $0<\alpha,\gamma<1,$ we denote by $\Psi_{\gamma,\alpha}$ the function given by $$\Psi_{\gamma,\alpha}(t,s):=\displaystyle\int_0^{\infty}\psi_{\gamma,0}(t,u)\psi_{\alpha,0}(s,u)\,du,\quad t,s>0.$$ Then $\Psi_{\gamma,\alpha}(t,s)=\Psi_{\alpha,\gamma}(s,t)$
, and  \begin{eqnarray*}\displaystyle\int_0^{\infty} e^{-\lambda t}\Psi_{\gamma,\alpha}(t,s)\,dt&=&s^{\alpha-1}E_{\alpha,\alpha}(-\lambda^{\gamma}s^{\alpha}), \qquad s>0,\cr
\int_0^{\infty} \int_0^\infty e^{-\mu s} e^{-\lambda t}\Psi_{\gamma,\alpha}(t,s)\,dt\, ds&=&{1\over \lambda^\gamma+ \mu^{\alpha}}
\end{eqnarray*}
for $\mu, \, \lambda>0.$
\end{corollary}
Note that
$$
\Psi_{{1\over 2},{1\over 2}}(t,s)={1\over 2\sqrt{\pi}(t+s)^{3\over 2}}, \qquad t,s>0.
$$

The following key lemma includes  the particular case $\alpha=\beta=\frac{1}{2}$ and $0<\eta$ proved in \cite[Lemma 1]{Miana}.
\begin{proposition}\label{convo} For $0<\alpha<1$,  $\beta\geq 0,$ and $0<\eta$, the following  identity holds $$\psi_{\alpha,\beta+\alpha\eta}(t,u)=\int_u^{\infty}g_{\eta}(s-u)\psi_{\alpha,\beta}(t,s)\,ds$$ for $t,u>0.$

\end{proposition}
\begin{proof} Note that $\psi_{\alpha,\beta}$ is a Laplace transformable function and locally integrable in two variables. We apply the Laplace transform in  variable $t$ to get in the right side
$$\int_0^{\infty}e^{-\lambda t}\int_u^{\infty}g_{\eta}(s-u)\psi_{\alpha,\beta}(t,s)\,ds\,dt=\lambda^{-\beta}\int_u^{\infty}g_{\eta}(s-u)e^{-\lambda^{\alpha}s}\,ds=\lambda^{-\beta}\lambda^{-\alpha\eta}e^{-\lambda^{\alpha}u},$$ with $\lambda>0$ where we have applied  Theorem \ref{propiedades} (ii) and \cite[Chapter II, (5.11)]{Miller}.

 In the left side, we also apply the Laplace transform in the variable $t$ to get that $$\int_0^{\infty}e^{-\lambda t}\psi_{\alpha,\beta+\alpha\eta}(t,u)\,dt=\lambda^{-(\beta+\alpha \eta)}e^{-\lambda^{\alpha}u},\qquad u>0,$$ with $\lambda>0,$ where we have used Theorem \ref{propiedades} (ii).
\end{proof}

\noindent {\it Remark.} For $u=0$ in Proposition \ref{convo}, we obtain the equality
$$
\psi_{\alpha,\beta+\alpha\eta}(t,0)=\int_0^\infty g_{\eta}(s)\psi_{\alpha,\beta}(t,s)\,ds=g_{\alpha\eta+\beta}(t), \qquad t>0,
$$
proved in Theorem \ref{propiedades} (vi).

Finally, we show an algebraic identity which satisfies functions $\psi_{\alpha,1-\alpha}$.

\begin{proposition} \label{integrales} Take $0<\alpha<1,$ $\beta>\alpha$ and $t,s>0$. Then \begin{itemize}
\item[(i)]$\displaystyle
 \left( \int_t^{t+s}  - \int_0^s\right)\frac{(t+s-r)^{\beta-\alpha-1}}{\Gamma(\beta- \alpha)}\psi_{\alpha,\beta-\alpha}(r,u)\, dr \\ \\ \qquad \qquad = \frac{\alpha}{\Gamma(1-\alpha)} \int_0^t \int_0^s \frac{\left(\psi_{\alpha,\beta-\alpha}(r_1,(\cdot))\ast \psi_{\alpha,\beta-\alpha}(r_2,(\cdot))\right)(u) }{(t+s-r_1-r_2)^{1+\alpha}}\,dr_1\, dr_2.
 $ \\

 \item[(ii)]$\displaystyle \psi_{\alpha,0}(t+s,u)=\frac{\alpha}{\Gamma(1-\alpha)} \int_0^t \int_0^s \frac{\left(\psi_{\alpha,0}(r_1,(\cdot))\ast \psi_{\alpha,0}(r_2,(\cdot))\right)(u) }{(t+s-r_1-r_2)^{1+\alpha}}\,dr_1\, dr_2.$
\end{itemize}
\end{proposition}
\begin{proof}Note that $\widehat{\psi_{\alpha,\beta-\alpha}}(r,(\cdot))(\lambda)=r^{\beta-1}E_{\alpha,\beta}(-\lambda r^\alpha)$, for $\beta\geq \alpha,$ $r>0$ and $\lambda\in \C$, see Proposition \ref{propiedades} (iii). Then we apply the Laplace transform in the variable $u$ in both equalities to get the identities \eqref{eq2.3} and \eqref{eq2.4}. The injectivity of Laplace transform allows us to finish the proof.
\end{proof}

To finish this section, we prove a subordination formula for functions $\psi_{\alpha, \beta}$ which expands some well-known results.

\begin{theorem} \label{bonito} For $0<\alpha, \delta<1$,  $\beta\ge \alpha $ and $ \delta \ge \gamma$, the following identity holds
$$
\psi_{\alpha\gamma, \beta-\alpha +\alpha(\delta-\gamma)}(t,s)=\int_0^\infty \psi_{\alpha, \beta-\alpha}(t,r)\psi_{\gamma, \delta-\gamma}(r,s)\,dr, \qquad t,s>0.
$$

\end{theorem}

\begin{proof} To show this theorem, we apply the Laplace transform in both variables $(t,s)$, the so called double Laplace transform, Fubini theorem, Theorem \ref{propiedades} (ii) and (iii) and finally formula (\ref{mittag}) to get that
\begin{eqnarray*}
\int_0^\infty\int_0^\infty e^{-\lambda t-\mu s}\int_0^\infty \psi_{\alpha, \beta-\alpha}(t,r)\psi_{\gamma, \delta-\gamma}(r,s)dr\,ds\,dt&=&\int_0^\infty \lambda^{\alpha-\beta} e^{-\lambda^\alpha r} r^{\gamma-1}E_{\gamma, \delta}(-\mu r^\gamma)\,dr\\&=& {\lambda^{\alpha\gamma- (\beta-\alpha +\alpha \gamma)}\over \lambda^{\alpha \gamma}+\mu}
\end{eqnarray*}
for $\Re \lambda, \Re \mu >0$. Due to Theorem \ref{propiedades} (iv) and  the uniqueness of the double Laplace transform (see for example \cite[p. 346]{DP}), we conclude the equality.
\end{proof}

\begin{remark}\label{cases}{\rm  In the case that $\beta=\delta=1,$ we obtain the formula (\ref{Mainardi}). For $\alpha=\beta$ and $\gamma=\delta,$ we get the following subordination formula for  stable L\'{e}vy processes
$$
f_{s, \alpha \gamma}(t)=\int_0^\infty f_{r, \alpha}(t)f_{s, \gamma}(r)\,dr, \qquad s,t>0,
$$
for $0<\alpha,\gamma<1.$ Finally for $\alpha=\gamma={1\over 2}$ and $\beta=\delta=1$, we obtain that
$$
\psi_{{1\over 4},{3\over 4 }}(t,s)={1\over \pi \sqrt{t}}\int_0^\infty e^{{-r^2\over 4t}-{s^2\over 4r}}{dr\over \sqrt{r}}, \qquad t,s>0,
$$
where we use the equality (\ref{gassuian})}.
\end{remark}

\section{Subordination principle and spectral inclusions for regularized resolvent families}

\setcounter{theorem}{0}
\setcounter{equation}{0}

In the following we consider that the operator $A$ is a densely defined closed linear operator on a Banach space $X$.
Let $\alpha,\beta>0.$ A family $\{S_{\alpha,\beta}(t)\}_{t> 0}\subset\mathcal{B}(X)$ is a $(g_{\alpha},g_{\beta})$-regularized resolvent family generated by $A$ if the following conditions are satisfied.
\begin{itemize}

\item[(a)] $S_{\alpha,\beta}(t)$ is strongly continuous for $t> 0$ and
$\displaystyle\lim_{t\to 0^+}\frac{S_{\alpha,\beta}(t)x}{g_\beta(t)}=x$ for all $x\in X;$ \\

\item[(b)] $S_{\alpha,\beta}(t)A\subset AS_{\alpha,\beta}(t)$, i.e.,  $S_{\alpha,\beta}(t)A(x)\subset AS_{\alpha,\beta}(t)x$ for $x\in D(A)$ and  $t>0.$

\item[(c)] The integral equation $$S_{\alpha,\beta}(t)x={t^{\beta-1}\over \Gamma(\beta)}x+A\left(\int_0^t{(t-s)^{\alpha-1}\over \Gamma(\alpha)}S_{\alpha,\beta}(s)x\,ds\right),$$ holds for $x\in X$ and $t>0 .$

\end{itemize}
This family of operators was formerly introduced  for general kernels $(a, k)$ in \cite[Definition 2.1]{Lizama}. The above definition is also considered for $\alpha>0$ and $\beta\geq 1$  in \cite{Chen}; for $0<\alpha=\beta<1$ in \cite{Mei-Pe-Zh13}; and
for $\alpha>0$ and $\beta=1$ in \cite[Definition 2.1]{Bazhlepaper}.

The reason why we do not consider the value of $S_{\alpha,\beta}(\cdot)$ at $0$ in the condition $(a)$ (compare with \cite[Definition 2.1, Condition (R1)]{Lizama}) is that the function $t\mapsto g_{\beta}(t)$ has a singularity at $0$ if $0<\beta<1.$

Let $ S: (0,\infty)\to {\mathcal B}(X) $ be a strongly continuous operator family such that $S(\cdot)x\in L^1_{loc}(\R_+,X),$ for any $x\in X.$ The operator family $\{S(t)\}_{t>0}$ is said Laplace-transformable if there exists $\omega\in \R$ such that the Laplace transform of
$S$ $$ \hat S (\lambda)x = \displaystyle
\int_0^\infty e^{- \lambda t} S(t)x\, dt, \qquad\Re \lambda>\omega $$ converges for $x\in X$, see for example \cite[Definition 3.1.4]{ABHN}.
If A generates a $(g_{\alpha},g_{\beta})$-regularized resolvent family $\{S_{\alpha,\beta}(t)\}_{t> 0}$ such that $S(\cdot)x\in L^1_{loc}(\R_+,X)$ for $x\in X$ and is Laplace transformable of parameter $\omega,$ we write $A\in \mathcal{C}^{\alpha,\beta}(\omega).$ We denote by $\mathcal{C}^{\alpha,\beta}:=\bigcup \{\mathcal{C}^{\alpha,\beta}(\omega);\ \omega\geq 0\}.$ The case of $(g_{\alpha},1)$-regularized resolvent families exponentially bounded, $\lVert S_{\alpha,1}(t)\rVert\leq Me^{wt}$ for $t>0,$ is considered in \cite[Definition 2.5]{Li} and \cite[Definition 2.4]{Bajlekova}, in this case, $\mathcal{C}^{\alpha}(\omega):=\mathcal{C}^{\alpha,1}(\omega).$

The next theorem characterizes the Laplace transform of $(g_\alpha, g_\beta)$-regularized family,   extends \cite[Theorem 3.11]{Chen} and  the proof is similar to the proof of \cite[Proposition 3.1]{Lizama}.

\begin{theorem}\label{caracterizacion} Let $\alpha,\beta > 0.$ Then $A\in\mathcal{C}^{\alpha,\beta}(\omega)$ if and only if $(\omega^{\alpha},\infty)\subset \rho(A)$ and there exists a strongly continuous function $S_{\alpha,\beta}(\cdot):(0,\infty)\to \mathcal{B}(X)$, locally integrable,
 $\displaystyle\lim_{t\to 0^+}\frac{S_{\alpha,\beta}(t)x}{g_\beta(t)}=x$ for all $x\in X,$ and
  Laplace transformable such that $$\int_0^{\infty}e^{-\lambda t} S_{\alpha,\beta}(t)x\,dt=\lambda^{\alpha-\beta}(\lambda^{\alpha}-A)^{-1}x,\qquad \lambda>\omega,$$ for all $x\in X.$ Furthermore, the family $\{S_{\alpha,\beta}(t)\}_{t>0 }$ is the $(g_{\alpha},g_{\beta})$-regularized resolvent family generated by $A.$
\end{theorem}

\begin{example}\label{canonical} {\rm For $0<\alpha<1$ and $\beta \ge \alpha$, the family $(\psi_{\alpha, \beta-\alpha}(t, (\cdot)))_{t>0}$ is a $(g_{\alpha},g_{\beta})$-regularized resolvent family generated by $A=-{d\over dt}$ on the Banach space $L^1(\R_+)$ (Theorem \ref{propiedades} (ii)); in particular the  stable L\'{e}vy processes $(f_{(\cdot),\alpha}(t))_{t>0}$ are $(g_{\alpha},g_{\alpha})$-regularized resolvent family. Similarly Mittag-Leffer functions $ (t^{\beta-1}E_{\alpha, \beta}(-t^\alpha(\cdot)))_{t\ge 0}$ are $(g_{\alpha},g_{\beta})$-regularized resolvent families generated by $Af(z)=-z f(z)$  on the Banach space $C_0(\R_+)$ (or ${\mathcal H}_0(\C_+)$), see formula (\ref{mittag}). These two families of functions are canonical examples of $(g_{\alpha},g_{\beta})$-regularized resolvent families.
}
\end{example}

Spectral inclusions for regularized resolvent families have been studied in certain particular cases. For example, spectral inclusions for $C_0$-semigroups($(1,1)$-regularized resolvent families in \cite[Chapter IV]{Nagel}) or cosine functions ($(t,1)$-regularized resolvent families in \cite{Nagy}) are well known. Finally, spectral inclusions  for $({t^{\alpha-1}\over \Gamma(\alpha)},1)$-regularized resolvent families ($\alpha>0$) are shown in \cite[Theorem 3.2]{Li-Zheng}.  We denote by $$m_{\alpha,\beta}^{a}(t):=t^{\beta-1}E_{\alpha,\beta}(a t^{\alpha}), \qquad t>0, \qquad a\in \C,$$
for $\alpha, \beta>0$. Note that $m_{\alpha,\beta}^{a}(t)=\widehat{\psi_{\alpha, \beta-\alpha}(t, (\cdot)) }(a)$ for $\beta>\alpha$ and $t>0$ (Theorem \ref{propiedades} (iii)).

\begin{lemma} Suppose that $\{S_{\alpha,\beta}(t)\}_{t>0}$ is a $(g_{\alpha},g_{\beta})$-regularized resolvent family generated by $A\in\mathcal{C}^{\alpha,\beta}(\omega).$ Then the following equalities hold.

\begin{equation}\label{EspecIncl2}
\int_0^{t}m_{\alpha,\alpha}^{a}(t-s)S_{\alpha,\beta}(s)(a-A)x\,ds=m_{\alpha,\beta}^{a}(t)x-S_{\alpha,\beta}(t)x,\qquad x\in D(A).
\end{equation}

\begin{equation}\label{EspecIncl1}
(a-A)\int_0^{t}m_{\alpha,\alpha}^{a}(t-s)S_{\alpha,\beta}(s)x\,ds=m_{\alpha,\beta}^{a}(t)x-S_{\alpha,\beta}(t)x,\qquad x\in X.
\end{equation}

\end{lemma}
\begin{proof} Take $x\in D(A).$ By \eqref{mittag} and Theorem \ref{caracterizacion} we have that \begin{displaymath}\begin{array}{l} \displaystyle\int_0^{\infty}e^{-\lambda t}(\int_0^{t}m_{\alpha,\alpha}^{a}(t-s)S_{\alpha,\beta}(s)(a-A)x\,ds)\,dt=\frac{\lambda^{\alpha-\beta}(\lambda^{\alpha}-A)^{-1}(a-A)x}{\lambda^{\alpha}-a} \\
\displaystyle\qquad=\frac{\lambda^{\alpha-\beta}}{\lambda^{\alpha}-a}x-\lambda^{\alpha-\beta}(\lambda^{\alpha}-A)^{-1}x=\displaystyle\int_0^{\infty}e^{-\lambda t}(m_{\alpha,\beta}^{a}(t)x-S_{\alpha,\beta}(t)x)\,dt,\end{array}\end{displaymath}
 and we obtain \eqref{EspecIncl2} by the uniqueness of Laplace transform. The proof of \eqref{EspecIncl1} follows since $A$ is a closed densely defined operator.
\end{proof}

For a closed operator $A,$ we denote by $\sigma(A),$ $\sigma_{p}(A),$ $\sigma_{r}(A)$ and $\sigma_{a}(A)$ the spectrum, point spectrum, residual spectrum and approximate point spectrum of $A$ respectively. For more details see \cite[Chapter IV]{Nagel}. The proof of the next theorem is inspired in the proof of \cite[Theorem 3.2]{Li-Zheng} and as some parts run parallel we skip them.

\begin{theorem} \label{spectral} Let $\{S_{\alpha,\beta}(t)\}_{t>0}$ be a $(g_{\alpha},g_{\beta})$-regularized resolvent family generated by $A\in\mathcal{C}^{\alpha,\beta}(\omega).$ Then we have \begin{itemize}
\item[(i)] $m_{\alpha,\beta}^{\lambda}(t)\in\sigma(S_{\alpha,\beta}(t)),$ for $\lambda\in\sigma(A).$
\item[(ii)] $m_{\alpha,\beta}^{\lambda}(t)\in\sigma_{p}(S_{\alpha,\beta}(t)),$ for $ \lambda\in\sigma_{p}(A).$
\item[(iii)] $m_{\alpha,\beta}^{\lambda}(t)\in\sigma_{a}(S_{\alpha,\beta}(t)),$ for $ \lambda\in\sigma_ {a}(A).$
\item[(iv)] $m_{\alpha,\beta}^{\lambda}(t)\in\sigma_{r}(S_{\alpha,\beta}(t)),$ for $\lambda\in\sigma_{r}(A).$
\end{itemize}

\end{theorem}
\begin{proof}  We define the family of operators $$\displaystyle B_{\lambda}(t)x:=\int_0^{t}m_{\alpha,\alpha}^{\lambda}(t-s)S_{\alpha,\beta}(s)\,ds,\qquad x\in X,$$ then $B_{\lambda}(t)\in\mathcal{B}(X)$ for $t>0$. (i) Suppose that $m_{\alpha,\beta}^{\lambda}(t)\in\rho(S_{\alpha,\beta}(t))$  and we denote by $Q(t):=(m_{\alpha,\beta}^{\lambda}(t)I-S_{\alpha,\beta}(t))^{-1}$ for $t>0$. By \eqref{EspecIncl1} and \eqref{EspecIncl2} we get that $$(\lambda-A)B_{\lambda}(t)Q(t)x=x,\qquad x\in X,\quad t>0,$$ and $$Q(t)B_{\lambda}(t)(\lambda-A)x=x,\qquad x\in D(A),\quad t>0.$$ Using that $Q(t)B_{\lambda}(t)=B_{\lambda}(t)Q(t)$ we have $(\lambda-A)^{-1}=B_{\lambda}(t)Q(t)\in\mathcal{B}(X),\quad t>0$ which implies $\lambda\in\rho(A).$ Parts (ii), (iii) and (iv) are similar to the proof of \cite[Theorem 3.2]{Li-Zheng}.
\end{proof}

The next theorem is the main one of this paper.

\begin{theorem}\label{main} Let $0<\eta_1\leq 2,$, $0<{\eta_2}$ and $\omega\geq 0.$ If $A\in \mathcal{C}^{\eta_1,\eta_2}(\omega)$ generates a $(g_{\eta_1},g_{\eta_2})$-regularized resolvent family $\{S_{\eta_1,\eta_2}(t)\}_{t> 0}$ then $A\in\mathcal{C}^{\alpha\eta_1,\alpha\eta_2+\beta}(\omega^{\frac{1}{\alpha}})$ generates the following $(g_{\alpha\eta_1},g_{\alpha\eta_2+\beta})$-regularized resolvent family \begin{equation}\label{formulasub}S_{\alpha\eta_1,\alpha\eta_2+\beta}(t)x:=\int_0^{\infty}\psi_{\alpha,\beta}(t,s)S_{\eta_1,\eta_2}(s)x\,ds,\qquad t>0,\quad x\in X,\end{equation}
for $0<\alpha<1$ and $\beta \ge 0.$ Moreover the following equality holds
\begin{equation}\label{formulasub2}S_{\alpha\eta_1,\alpha\eta_2+\beta}(t)x=(g_{\beta}*S_{\alpha\eta_1,\alpha\eta_2})(t)x,\qquad t>0, \quad x\in X,
\end{equation}
for $\beta>0.$

\end{theorem}
\begin{proof} First we show  that $\{S_{\alpha\eta_1,\alpha\eta_2+\beta}(t)\}_{t>0}$ is Laplace transformable of parameter $\omega^{\frac{1}{\alpha}}.$
 Note that using Proposition \ref{propiedades} (ii) we get that \begin{eqnarray*}\int_0^{\infty} e^{-\lambda t} S_{\alpha\eta_1,\alpha\eta_2+\beta}(t)x\, dt&=&\int_0^{\infty} (\int_0^{\infty}e^{-\lambda t} \psi_{\alpha,\beta}(t,s)\,dt) S_{\eta_1,\eta_2}(s)x\,ds\\&=&\lambda^{-\beta}\int_0^{\infty}e^{-\lambda^{\alpha s}}S_{\eta_1,\eta_2}(s)x\,ds
=\lambda^{\alpha\eta_1-(\alpha\eta_2+\beta)}(\lambda^{\alpha\eta_1}-A)^{-1}x,\end{eqnarray*} for $\lambda> \omega^{\frac{1}{\alpha}},$ and $(\omega^{\eta_1},\infty)=((\omega^{\frac{1}{\alpha}})^{\alpha\eta_1},\infty)\subset \rho(A).$

The family $\{S_{\alpha\eta_1,\alpha\eta_2+\beta}(t)\}_{t>0}$ is strongly continuous on $(0, \infty)$ and now we prove the strong continuity at the origin. Let $x\in X,$ then \begin{eqnarray*}
&\quad&\frac{\lVert S_{\alpha\eta_1,\alpha\eta_2+\beta}(t)x-g_{\alpha\eta_2+\beta}(t)x\rVert}{g_{\alpha\eta_2+\beta}(t)} \leq\displaystyle\int_0^{\infty}\frac{\psi_{\alpha,\beta}(t,s)}{g_{\alpha\eta_2+\beta}(t)}\lVert S_{\eta_1,\eta_2}(s)x-g_{\eta_2}(s)x\rVert\,ds \\
&\quad&\qquad\qquad=\displaystyle\Gamma(\alpha\eta_2+\beta)\int_0^{\infty}t^{\alpha-\alpha\eta_2} W_{-\alpha,\beta}(-u)\lVert S_{\eta_1,\eta_2}(ut^{\alpha})x-g_{\eta_2}(ut^{\alpha})x\rVert\,du \\
&\quad&\qquad\qquad=\displaystyle\Gamma(\alpha\eta_2+\beta)\int_0^{\infty}\frac{g_{\eta_2}(u)W_{-\alpha,\beta}(-u)}{g_{\eta_2}(ut^{\alpha})}\lVert S_{\eta_1,\eta_2}(ut^{\alpha})x-g_{\eta_2}(ut^{\alpha})x\rVert\,du, \\
\end{eqnarray*} where we have used Theorem \ref{propiedades} (vi). We apply the dominated convergence theorem to the above term and  we conclude that
$$\frac{\lVert S_{\alpha\eta_1,\alpha\eta_2+\beta}(t)x-g_{\alpha\eta_2+\beta}(t)x\rVert}{g_{\alpha\eta_2+\beta}(t)}  \to 0, \qquad t\to 0^+,$$ since $S_{\eta_1,\eta_2}(t)$ is a $(g_{\eta_1},g_{\eta_2})$-regularized resolvent family and Proposition \ref{W} (i).
Finally, by Theorem \ref{caracterizacion}, we obtain that the family $\{S_{\alpha\eta_1,\alpha\eta_2+\beta}(t)\}_{t>0}$ is a  $(g_{\alpha\eta_1},g_{\alpha\eta_2+\beta})$-regularized resolvent family generated by $A$. The proof of the  equality (\ref{formulasub2}) is a straightforward consequence of Theorem \ref{propiedades} (v).
\end{proof}

\begin{remark}\label{others} {\rm  Note that the above subordination theorem extends some subordination results which have been considered in this and previous papers: \begin{enumerate}
\item[(i)] Now, we consider the family of functions $(\psi_{\gamma, \delta-\gamma}(t, (\cdot)))_{t>0}$ which is a $(g_{\gamma},g_{\delta})$-regu\-larized resolvent family for $0<\gamma<1$ and $\delta\ge \gamma$ (see Example \ref{canonical}). Then we apply Theorem \ref{main} for $0<\alpha<1$ and $\beta>0$, and we obtain the formula
$$
\psi_{\alpha\gamma, \beta +\alpha(\delta-\gamma)}(t,(\cdot))=\int_0^\infty \psi_{\alpha, \beta}(t,s)\psi_{\gamma, \delta-\gamma}(s,(\cdot))\,ds, \qquad t>0,
$$
which is shown in Theorem \ref{bonito}.

\item[(ii)] If $0<\eta_1\leq 2$ and $\eta_2=1$ in Theorem \ref{main}, we retrieve the subordination principle for $(g_{\alpha\eta_1},1)$-regularized resolvent families given in   \cite[Theorem 3.1]{Bazhlepaper} and \cite[Theorem 3.1]{Bajlekova} for $0<\alpha<1$ and $\beta=1-\alpha$.

\item[(iii)] By \cite[Theorem 2.8 (i)]{Kostic}, given a $(g_{\eta_1,\eta_2})$-regularized family generated by $A,$ we obtain a $(g_{\gamma_1},g_{\gamma_2})$-regularized resolvent family generated by $A,$ where $0<\gamma_1<\eta_1$ and $\gamma_2=1- {\gamma_1\over \eta_1}+ {\gamma_1\over \eta_1}\eta_2$. This is  a particular case of  Theorem \ref{main}  for $\alpha={\gamma_1\over \eta_1}$ and $\beta= 1-{\gamma_1\over \eta_1}$.

\end{enumerate}
}
\end{remark}

\begin{corollary}\label{corollary} {\rm
Let $0<\eta\leq 2,$ and $0<\alpha<1$. If $\{S_{\eta,\eta}(t)\}_{t> 0}$ is a $(g_{\eta},g_{\eta})$-regularized resolvent family generated by $A\in\mathcal{C}^{\eta,\eta}(\omega),$ then $A\in\mathcal{C}^{\alpha\eta,\alpha\eta}(\omega^{\frac{1}{\alpha}})$ generates the following $(g_{\alpha\eta},g_{\alpha\eta})$-regularized resolvent family $$S_{\alpha\eta,\alpha\eta}(t)x:=\int_0^{\infty}\psi_{\alpha,0}(t,s)S_{\eta,\eta}(s)x\,ds,\qquad t>0,\quad x\in X.$$ In particular, we have the following remarkable particular cases.

    \begin{itemize}
    \item[(i)] If $A$ generates a $C_0$-semigroup $\{T(t)\}_{t>0}$, i.e. a $(g_1,g_1)$-regularized resolvent family, satisfying $\lVert T(t)\rVert\leq Me^{\omega t},$ $t\geq 0,$ then $A\in\mathcal{C}^{\alpha,\alpha}(\omega^{\frac{1}{\alpha}})$ generates the following $(g_{\alpha},g_{\alpha})$-regularized resolvent family $$S_{\alpha,\alpha}(t)x:=\int_0^{\infty}\psi_{\alpha,0}(t,s)T(s)x\,ds,\qquad t>0,\quad x\in X.$$

\item[(ii)]Let $\beta\in (1,2]$ and $\{S_{\beta,\beta}(t)\}_{t> 0}$ be a $(g_{\beta},g_{\beta})$-regularized resolvent family generated by $A\in\mathcal{C}^{\beta,\beta}(\omega).$ Then $A\in\mathcal{C}^{1,1}(\omega^{\beta})$ generates the following $C_0$-semigroup, $$T(t)x:=\int_0^{\infty}\psi_{\frac{1}{\beta},0}(t,s)S_{\beta,\beta}(s)x\,ds, \qquad t>0,\quad x\in X.$$
In the case that $A$ is the generator of a exponentially bounded sine function $\lVert S(t)\rVert\leq Me^{\omega t},$ $(t\geq 0)$  (see details in \cite[Section 3.15]{ABHN}), i.e. $(S(t))_{t>0}$ is a $(g_2,g_2)$-regularized resolvent family, then $A$ generates the following $C_0$-semigroup, $$T(t)x:=
     {1\over 2\sqrt{\pi}}t^{-3\over 2}\int_0^{\infty}se^{-s^2\over 4t} S(s)xds,\qquad t>0,\quad x\in X,$$ with exponential bound $\omega^{2},$ which is proved in \cite[Theorem 5.2]{Arendt} and see also \cite[Proposition 2.5 and Theorem 2.6]{Keyantuo}.
\end{itemize}
}
\end{corollary}
\section{Applications, examples and final comments}

\setcounter{theorem}{0}
\setcounter{equation}{0}

In this last section we present some applications of our results to fractional Cauchy problems and classical semigroups in Lebesgue spaces. Now we recall some basic definitions on fractional Cauchy problems. Let $f\in C_{c}^{(\infty)}(\R_+;X),$ we call Riemann-Liouville fractional integral of $f$  of order $\gamma>0$, $I^{-\gamma}f$  to the function given by $$I^{-\gamma}f(t):=g_{\gamma}*f(t)=\int_0^t{(t-s)^{\gamma-1}\over \Gamma(\gamma)}f(s)\,ds,\qquad t\geq 0,$$ and the Riemann-Liouville fractional derivative of $f$  of order $\gamma>0$ is given by $$_RD_t^{\gamma}f(t):=\frac{d^n}{dt^n}\left(I^{-(n-\gamma)}f\right)(t),\qquad t\geq 0,$$ and with $n=[\gamma] +1.$ Also, we consider the Caputo fractional derivative of $f$ of order $\gamma>0,$ $$_CD_t^{\gamma}f(t):=g_{n-\gamma}*f^{(n)}(t)=I^{-(n-\gamma)}f^{(n)}(t),\qquad t\geq 0,$$ with $n=[\gamma] +1,$  see for example \cite{Mainardi, Miller} and \cite[Section 4]{Li}. Note that in the above definitions, the function $f$ can be taken in a larger space than $C_{c}^{(\infty)}(\R_+)$ where the definitions make sense.

\subsection{Fractional powers in fractional Cauchy problems}

 The next results extend \cite[Theorem 4.9 (a), (c)]{Li}: for $0<\alpha<1,$ the solutions of Caputo fractional problems are $(g_{\alpha},1)$-regularized resolvent families (see \cite[Definition 2.3]{Bajlekova}) which may be obtained by integration from $(g_{\alpha},g_{\alpha})$-regularized resolvent families by \eqref{formulasub2}. Note that these are solutions of Riemann-Liouville fractional problems, see \cite[Theorem 1.1]{Mei-Pe-Zh13}.

\begin{theorem}\label{cauchy}Let $\alpha, \gamma\in (0,1)$ and $A$ the generator of a uniformly bounded $C_0$-semigroup $\{T(t)\}_{t>0}\subset {\mathcal B}(X).$

\begin{itemize}
\item[(i)] The fractional Cauchy problem
\begin{equation}\label{cauchy1}
\left\{\begin{array}{ll}
_RD_t^{\alpha}v(t)=Av(t),&t>0, \\ \\
(g_{1-\alpha}*v)(0)=x\in D(A),&
\end{array} \right.
\end{equation}
is well-posed and its unique solution is given by $$v(t)=\displaystyle\int_0^{\infty}\psi_{\alpha,0}(t,s)T(s)x\,ds,\qquad t>0.$$

\item[(ii)] The fractional Cauchy problem
\begin{equation*}
\left\{\begin{array}{ll}
_RD_t^{\gamma}v(t)=-(-A)^{\alpha}v(t),&t>0, \\ \\
(g_{1-\gamma}*v)(0)=x\in D(A),&
\end{array} \right.
\end{equation*}
is well-posed and its unique solution is given by $$v(t)=\displaystyle\int_0^{\infty}\Psi_{\gamma,\alpha}(t,s)T(s)x\,ds,\qquad t>0.$$


\end{itemize}
\end{theorem}
\begin{proof}
(i) By Corollary \ref{corollary} (i), the operator $A$ generates a $(g_\alpha, g_\alpha)$-regularized resolvent family which provides the solution of the fractional Cauchy problem (\ref{cauchy1}), see \cite[Theorem 1.1]{Mei-Pe-Zh13}. (ii) By \cite[Chapter IX]{Yosida}, the operator $-(-A)^\alpha$ generates a $C_0$-semigroup $\{T^{(\alpha)}(t)\}_{t>0}$ given by (\ref{yosida}). By  part (i), Fubini theorem and the definition of $\Psi_{\gamma,\alpha}$ in Corollary \ref{PsiCapital}, we conclude the proof.
\end{proof}

\subsection{Convolution semigroups on $L^p(\R^n)$} Let $\{T(t)\}_{t\geq 0}$ be a uniformly bounded convolution semigroup in $\R^n$ generated by $A,$ that is,
$$T(t)f(x)=(k_{t}*f)(x)=\int_{\R^n}k_{t}(x-y)f(y)\,dy, \qquad t>0, x\in \R^n.$$ Two well known examples of convolution semigroups are Gaussian and Poisson semigroups, $g_{t}(x)=\frac{1}{(4\pi t)^{\frac{n}{2}}}e^{-\frac{|x|^2}{4t}}$ and $p_{t}(x)=\frac{\Gamma(\frac{n+1}{2})}{\pi^{\frac{n+1}{2}}}\frac{t}{(t^2+|x|^2)^{\frac{n+1}{2}}},$ whose generators are the Laplacian $\Delta$ and $-(-\Delta)^{\frac{1}{2}}$ respectively, see for example \cite[Chapter IX]{Yosida}. Then by Theorem \ref{cauchy} (i), the solution of the Riemann-Liouville fractional diffusion problem of order $0<\alpha<1$  \begin{equation*}
\left\{\begin{array}{ll}
_RD_t^{\alpha}u(t,x)=A u(t,x),&t>0, \\ \\
(g_{1-\alpha}*u(\cdot,x))(0)=f(x),&
\end{array} \right.
\end{equation*}
is given by $$u(t,x)=\displaystyle\int_{\R^n}\biggl( \int_0^{\infty}\psi_{\alpha,0}(t,s)k_{s}(x-y)\,ds \biggr)f(y)\,dy.$$ In addition, to the particular case of the Laplacian, we obtain the solution of the Caputo fractional diffusion problem of order $0<\alpha<1$, \begin{equation*}
\left\{\begin{array}{ll}
_CD_t^{\alpha}v(t,x)=\Delta v(t,x),&t>0, \\ \\
v(0,x)=f(x),&
\end{array} \right.
\end{equation*} (considered in \cite[Example 4.13]{Li}) is given by $v(t,x)=\int_0^{t}g_{1-\alpha}(t-s)u(s,x)\,ds,$ for $t>0.$

\subsection{Multiplication families on $C_0(\R^n)$} Let $\{T(t)\}_{t\geq 0}$ be a multiplication semigroup in $\R^n$ generated by $q(x),$ that is, $$T(t)f(x)=e^{tq(x)}f(x), \qquad x\in \R^n, \quad t>0.$$ Some examples are $q(x)=-4\pi^2|x|^2,\,-2\pi|x|,\,-\log(1+4\pi^2|x|^2),$ treated in \cite{Gale}. Then by Theorem \ref{cauchy} (i), the solution of the Riemann-Liouville fractional diffusion problem of order $0<\alpha<1$  \begin{equation*}
\left\{\begin{array}{ll}
_RD_t^{\alpha}u(t,x)=q(x)u(t,x),&t>0, \\ \\
(g_{1-\alpha}*u(\cdot,x))(0)=f(x),&x\in \R^n,
\end{array} \right.
\end{equation*}
is given by $$u(t,x)=\int_0^{\infty}\psi_{\alpha,0}(t,s)e^{sq(x)}f(x)\,ds=t^{\alpha-1}E_{\alpha,\alpha}(q(x)t^{\alpha})f(x),\qquad t>0,\quad x\in \R^n,$$ where we have applied Theorem \ref{propiedades} (iii). Even more, by Theorem \ref{cauchy} (ii), the solution of the Riemann-Liouville fractional diffusion problem of order $0<\alpha,\gamma<1$  \begin{equation*}
\left\{\begin{array}{ll}
_RD_t^{\gamma}u(t,x)=-(-q(x))^{\alpha}u(t,x),&t>0, \\ \\
(g_{1-\gamma}*u(\cdot,x))(0)=f(x),&
\end{array} \right.
\end{equation*}
is given by $$u(t,x)=\int_0^{\infty}\Psi_{\gamma,\alpha}(t,s)e^{sq(x)}f(x)\,ds=t^{\gamma-1}E_{\gamma,\gamma}(-(-q(x))^{\alpha}t^{\gamma})f(x),\qquad t>0,$$ where we have applied Theorem \ref{propiedades} (ii) and (iii).

\subsection*{Conflict of Interests} The authors declare that there is no conflict of interests regarding the publication of this paper.

\subsection*{Acknowledgements} We thank Prof. Carlos Lizama for some comments and nice ideas that have contributed to improve the final version of this publication.

\end{document}